\documentclass[11pt]{article}

\usepackage[margin=1in]{geometry}
\usepackage{amsthm}
\usepackage{graphicx}
\usepackage{amsmath,amssymb}
\usepackage{booktabs}
\usepackage{algorithm}
\usepackage{algpseudocode}
\usepackage{url}
\usepackage{hyperref}

\newcommand{\h}[1]{#1}

\newtheorem{theorem}{Theorem}

\theoremstyle{remark}
\newtheorem{remark}{Remark}
\theoremstyle{plain}
\newtheorem{assumption}{Assumption}

\newcommand{\CVaR}{\operatorname{\mathsf{CVaR}}}

\newcommand{\cbar}{\bar{\boldsymbol{c}}}
\newcommand{\bA}{\boldsymbol{A}}
\newcommand{\bb}{\boldsymbol{b}}
\newcommand{\bx}{\boldsymbol{x}}
\newcommand{\by}{\boldsymbol{y}}
\newcommand{\br}{\boldsymbol{r}}
\newcommand{\bd}{\boldsymbol{d}}
\newcommand{\bl}{\boldsymbol{l}}

\newcommand{\calT}{\mathcal{T}}

\begin{document}

\title{Euler Cutting Planes for Fractional CVaR Programming}

\author{Pranay Seshadri \and Hubert Dzieciol \and Zhe Feng \and Alan Tua\\[0.5ex]
  Ki Insurance, London, United Kingdom\\
  \texttt{pranay.seshadri@ki-insurance.com}}

\date{}

\maketitle

\begin{abstract}
  We maximise expected reward per unit of \emph{marginal} tail risk, a reward-to-CVaR ratio against a baseline. \h{The standard Dinkelbach route re-solves a Rockafellar--Uryasev LP or QP with one auxiliary variable per tail scenario at each multiplier update. Our method eliminates the multiplier iteration and keeps scenario auxiliaries out of the master, while still evaluating tail risk to generate cuts.} A Charnes--Cooper/Schaible homogenisation removes the reward/risk multiplier, and the remaining convex constraint is separated on demand: each cut coefficient is an Euler allocation vector, read off a ranking and an average. The result is a single convex program, finitely terminating at the global optimum, with size independent of $S$ --- an LP in $M{+}1$ variables, or a second-order cone program in $M{+}2$ for a concave quadratic numerator. The method runs $20$--$100\times$ faster than the exact Dinkelbach formulation and only needs single-digit cut counts.
\end{abstract}


\section{Introduction}
\label{sec:intro}
A recurring question in risk-constrained
 decision-making boils down to efficiency: which decisions deliver the most expected reward per unit of risk. This paper considers problems of the form
\begin{equation}
\begin{aligned}
  \max_{\bx}\; &
  \frac{f\left( \bx \right) }{\Delta\CVaR_\alpha(\bx)}, \; \; \; \textrm{where}
  \qquad \\
  f(\bx) & \;=\; \br^\top\bx - \tfrac12\,\bx^\top\boldsymbol{\Sigma}\,\bx, \\
  \qquad \boldsymbol{\Sigma} & \succeq 0, \; \; \; \textrm{and} \\
  \Delta\CVaR_\alpha(\bx) & = \CVaR_\alpha(\bb + \bA\bx) - \CVaR_\alpha(\bb).
  \label{eq:ratio}
\end{aligned}
\end{equation}
Here the linear term in $f(\bx)$ represents the expected reward, and the quadratic term covers a variance or cost penalty. The optimisation variable $\bx$ is a per-decision fraction, i.e., $\bx \in \left[0, \mathbf{1}\right]^M$ capturing how much of the asset / risk / policy we want. In the denominator, $\bA \in \mathbb{R}^{S\times M}$ is a scenario loss matrix ($M$ decisions, $S$ scenarios), $\bb \in \mathbb{R}^S$ is the base loss vector. The $\Delta\CVaR_\alpha(\bx)$ term is the \emph{marginal} CVaR: the increase in tail risk caused by the decisions, relative to a fixed baseline state. \h{Even when $\boldsymbol{\Sigma}=0$, the baseline can give marginal CVaR affine pieces with nonzero offsets, so the scale of $\bx$ does not generally cancel from the ratio; the homogenisation below still applies.} This structure --- a concave numerator over a convex, piecewise-linear denominator --- arises across risk-constrained engineering, insurance and finance. In all of these the scenario count $S$ is large and fixed by a generator, while the number of decisions $M$ is moderate.

As \eqref{eq:ratio} is a fractional program, the standard solution is nested and involves Dinkelbach's algorithm \cite{Dinkelbach1967}, where an outer loop iterates the reward/risk multiplier $\lambda$, and for each $\lambda$ an inner solver minimises $\lambda\ \Delta\CVaR_\alpha(\bx) - f\left( \bx \right)$. \h{With CVaR expanded through the Rockafellar--Uryasev (R--U) linearisation \cite{RockafellarUryasev2000}, each LP/QP inner solve carries $S$ auxiliary variables and is repeated at successive multiplier values.}

\h{We eliminate the multiplier iteration and R--U scenario auxiliaries by solving one cutting-plane convex program with on-demand risk cuts.} The resulting proposed algorithm, which we call the Euler Cutting-Plane (ECP) method, is a single LP in $M{+}1$ variables when there is no quadratic term in $f\left( \bx \right)$ and a $M{+}2$ variables SOCP when $\boldsymbol{\Sigma}\neq0$ with cuts added until the active tail stabilises. It terminates finitely at the global optimum (Theorem~\ref{thm:ccecp}). The scenario count enters only through an $O(MS)$ loss update and an $O(S)$ top-$k$ ranking; the master problem itself never grows with $S$, distinguishing the approach from the classical R-U linearisation.


Two ingredients make this possible: a Charnes--Cooper/Schaible homogenisation \cite{CharnesCooper1962,Schaible1983} that eliminates the Dinkelbach multiplier, and on-demand Euler cuts --- supporting hyperplanes of the CVaR denominator whose coefficients are tail conditional means, available in closed form from a sort and an average. Neither is new in isolation, and the cut is in essence F\'abi\'an's \cite{Fabian2008} simplification of the K\"unzi-Bay--Mayer representation, restated as equations (33)--(34) of \cite{FabianMitraRoman2011}. \h{Their composition in \eqref{eq:homcut} is the contribution: homogenisation preserves the tail-set separation oracle, removes the multiplier, and gives a finitely terminating cutting-plane method without scenario auxiliaries.}

The remainder of this paper is organised as follows: Section~\ref{sec:prelim} fixes the empirical CVaR, its subgradient (the Euler allocation vector), and the cutting-plane representation we build on. 
Section~\ref{sec:related} reviews the related work and positions the contribution in this manuscript precisely.
Section~\ref{sec:ccecp} derives ECP and proves finite termination, and Section~\ref{sec:qpnum} extends it to concave quadratic numerators.
Section~\ref{sec:experiments} reports experiments, all for the linear numerator $\boldsymbol{\Sigma}=0$.
Section~\ref{sec:conclusion} concludes.

\section{Preliminaries}
\label{sec:prelim}
The empirical CVaR at confidence level $\alpha$ of a loss vector
$\bl \in \mathbb{R}^S$ is the mean of its worst $k$ entries,
\begin{equation}
  \CVaR_\alpha(\bl) = \frac{1}{k}\sum_{s \in \calT(\bl)} l_s,
  \qquad k = (1-\alpha)S,
  \label{eq:cvar}
\end{equation}
where $\calT(\bl)$ indexes the $k$ largest components. We assume throughout that $1/(1-\alpha)S$ is an integer; the empirical CVaR is then exactly the tail mean.

As a function of the fractional contribution of each decision $\bx$, $\CVaR_\alpha(\bb+\bA\bx)$ is convex and piecewise linear: it is the pointwise maximum, over all size-$k$ scenario subsets $\calT$, of the affine functions $\frac1k\sum_{s\in\calT}(\bb_s + (\bA\bx)_s)$. On the polyehdral, active tail set, $\calT$, the function's gradient is \emph{Euler allocation vector}, given by
\begin{equation}
  (\cbar_{\calT})_i := \frac{1}{k}\sum_{s\in\calT} (\bA)_{si},
  \qquad i=1,\ldots,M
  \label{eq:euler}
\end{equation}
This is the conditional mean of the loss columns over the tail scenarios. Tasche
\cite{Tasche2008} calls this tail conditional mean the per-unit Euler capital
contribution to expected shortfall (his equation (3.11b)); Denault
\cite{Denault2001} calls the stacked vector of such contributions an allocation,
following which we call \eqref{eq:euler} the \emph{Euler allocation vector}.
Giving $\calT_0 = \calT(\bb)$ for the baseline tail, the marginal CVaR admits the global lower bound
\begin{align}
  \Delta\CVaR_\alpha(\bx) \;\ge\; \cbar_{\calT}^\top \bx + C(\calT),
  \qquad \\ 
  C(\calT) \mathrel{:=} \tfrac1k\big(\textstyle\sum_{s\in\calT} b_s -
  \sum_{s\in\calT_0} b_s\big) \le 0,
  \label{eq:support}
\end{align}
for every tail set $\calT$, with equality when $\calT$ is the active tail at $\bx$. \h{The baseline offset $C(\calT)$ belongs to the affine piece; the inequality is the cut.} As $\calT_0$ maximises $\sum_{s\in\calT} b_s$, the base cut has $C(\calT_0)=0$.



\section{Related Work}
\label{sec:related}

The standard linear-programming treatment of CVaR is the lifting formulation of
Rockafellar and Uryasev \cite{RockafellarUryasev2000}, later developed for
portfolio optimisation in \cite{RUZ2006}: a VaR variable and one
shortfall variable per scenario convert the empirical CVaR objective or
constraint into a finite LP. This representation is exact, but its dimension
grows with the scenario count. A complementary line of work avoids carrying all
scenario auxiliaries in the master problem by using cutting-plane
representations. Klein Haneveld and Van der Vlerk \cite{KleinHaneveldVlerk2006}
developed such representations
for integrated chance constraints, which are closely related to SSD and tail-risk
constraints; K\"unzi-Bay and Mayer \cite{KunziBayMayer2006} obtained the
analogous representation independently for CVaR minimisation, applying it to the
objective function where \cite{KleinHaneveldVlerk2006} applied it to a
constraint. Roman, Darby-Dowman and Mitra
\cite{RomanDarbyDowmanMitra2006} introduced a multi-objective portfolio model for second order stochastic dominance (SSD) efficient portfolios in which tail/CVaR functionals enter the objectives via
the Rockafellar--Uryasev CVaR formula; subsequent work by F\'abi\'an, Mitra and
Roman \cite{FabianMitraRoman2011} processed this class of SSD/CVaR models using
cutting-plane representations. Hofmann \cite{Hofmann2014} applies the same
cutting-plane philosophy directly to profit maximisation with CVaR constraints,
adding violated risk cuts instead of solving the full Rockafellar--Uryasev
reformulation. All of these treat CVaR in an objective or a constraint; here it is the denominator of a ratio.

As mentioned in the Introduction~\ref{sec:intro}, there are two structural facts that permit our approach. First, \eqref{eq:ratio} is a nonlinear fractional program, and it is so for two independent reasons --- one in the denominator, one in the numerator. Consider the linear numerator $\boldsymbol{\Sigma}=0$ first. Were the denominator the \emph{unshifted} $\CVaR_\alpha(\bA\bx)$, both parts of the ratio would be positively homogeneous of degree one, the ratio itself homogeneous of degree zero, and the problem would collapse by normalisation to
$\max\{\br^\top\bx:\CVaR_\alpha(\bA\bx)\le1\}$ --- with no multiplier and no
transformation required. The fixed baseline $\bb$ destroys that homogeneity: the
affine pieces of $\Delta\CVaR_\alpha$ carry distinct intercepts $C(\calT)\le0$
(Section~\ref{sec:prelim}), the scale of $\bx$ does not cancel, and a Dinkelbach
loop becomes the standard recourse. When $\boldsymbol{\Sigma}\neq0$ the cancellation fails a
second time, now in the numerator: $f$ is not positively homogeneous, so the
scale of $\bx$ does not cancel even where the denominator is homogeneous, and the
ratio strictly decreases along any ray on which the base tail stays active
(see Remark~\ref{rmk:edge}(iii)). Either departure alone suffices to rule out the
normalisation shortcut.

Second, the Charnes--Cooper substitution removes the multiplier, and does so at no cost in the separation oracle, because homogenisation \emph{commutes with separation}: the perspective of a max-of-affines is the max of the same affine pieces with their intercepts scaled by $t$. Thus, the homogenised denominator constraint is again a finite family of \emph{linear} inequalities indexed by tail sets, whose most-violated member at $(\by,t)$ is found by ranking the augmented loss at the de-homogenised iterate $\bx=\by/t$. Two consequences follow.

The homogenised feasible region is a polyhedron whose facets are indexed by the \emph{same} finite set as in the unhomogenised problem, so Kelley's method \cite{Kelley1960} terminates \emph{finitely} here rather than merely converging (Theorem~\ref{thm:ccecp}); and the cut coefficients remain the Euler allocation vectors that practitioners already compute. Note that homogenising a concave-over-convex ratio through the perspective is textbook fractional programming
\cite{Schaible1983,BoydVandenberghe2004}.

Against this, the three nearest lines of work each stop one step short. Kapsos et
al.\ \cite{KapsosZymler2014} homogenise a reward-to-risk ratio with the same
Charnes--Cooper scaling variable used here, but they \emph{lift} the denominator
before homogenising: one shortfall variable per sample, with the lifted
denominator normalised to one. Charnes--Cooper is then applied to a ratio of
affine functions, and the resulting LP carries a variable and a row per scenario,
so its dimension grows with the sample count; there is no separation oracle, and
hence no finite-termination argument over tail sets. We instead homogenise the
max-of-affines directly and generate its facets on demand. The cutting-plane
papers above supply that oracle --- indeed \cite{Fabian2008,FabianMitraRoman2011}
supply essentially the cut we use --- but have no ratio to homogenise, so the
question of whether homogenisation disturbs the cut family does not arise for
them. Barros and Frenk \cite{BarrosFrenk1995} do combine fractional programming
with cutting planes, but in the \emph{parametric} space, and their result is an
identification rather than a removal: they show that their cutting-plane
algorithm \emph{reduces to} the classical Dinkelbach-type procedure for
generalized fractional programs, thereby giving that procedure a geometric
interpretation. Their cuts approximate
$\lambda\mapsto\max_{\bx}\{f(\bx)-\lambda\,\Delta\CVaR_\alpha(\bx)\}$ inside a
root-finding iteration, so the multiplier survives; here it is absent from the
start and the cuts accumulate monotonically into one program. We are not aware of
prior work that generates CVaR cuts for a homogenised fractional program.


\section{The Euler Cutting-Plane Method}
\label{sec:ccecp}

\subsection{Homogenisation}
Let
\begin{equation}
  t = \frac{1}{\Delta\CVaR_\alpha(\bx)} > 0,
  \qquad \by = t\,\bx,
\end{equation}
so that the objective of \eqref{eq:ratio} becomes the \emph{perspective} of $f$,
\begin{equation}
  \frac{f(\bx)}{\Delta\CVaR_\alpha(\bx)} \;=\; t\,f(\by/t)
  \;=\; \br^\top\by \;-\; \frac{\by^\top\boldsymbol{\Sigma}\by}{2\,t}.
  \label{eq:persp}
\end{equation}
Two properties then emerge. It is jointly \emph{concave} in
$(\by,t)$ on $t>0$, because $\by^\top\boldsymbol{\Sigma}\by/t$ is quadratic-over-linear, so
the master stays a convex program solved to global optimality at every cut; and
it is positively homogeneous of degree one, as perspectives always are, which is
what forces the relaxed denominator constraint to bind at the optimum and so
keeps the recovery $\bx^\star=\by^\star/t^\star$ valid
(Remark~\ref{rmk:edge}(iii)). The box scales with $t$: $\underline{\bx} \le \bx \le \overline{\bx}$ becomes
$\underline{\bx}\,t \le \by \le \overline{\bx}\,t$, so a nonzero floor such as a
minimum policy, or risk size, costs only one further linear constraint.
Throughout, bold $\overline{\bx},\underline{\bx}\in\mathbb{R}^M$ denote the vector
bounds and plain $\bar x,\underline x$ their common component when every decision
shares one cap or floor, so that $\overline{\bx}=\bar x\,\boldsymbol1$. The denominator constraint
$\Delta\CVaR_\alpha(\bx) \le 1/t$ homogenises the same way: the perspective of
the max-of-affines \eqref{eq:support} is the max of the same affine pieces with
their intercepts scaled by $t$ (\cite{BoydVandenberghe2004}, p.~89), hence the
finite family of \emph{linear} inequalities
\begin{equation}
  \cbar_{\calT}^\top \by + C(\calT)\, t \;\le\; 1
  \qquad\text{for every tail set } \calT.
  \label{eq:homcut}
\end{equation}
Collecting these gives the homogenised master
\begin{equation}
\begin{aligned}
  \max_{\by,\,t}\; \br^\top\by - \frac{\by^\top\boldsymbol{\Sigma}\by}{2\,t} \\
  \text{s.t.} \quad \cbar_{\calT}^\top \by + C(\calT)\,t \le 1 \;\;\forall \calT,\quad \\
  \underline{\bx}\,t \le \by \le \overline{\bx}\,t,\quad t \ge 0,
  \label{eq:homfull}
\end{aligned}
\end{equation}
over $(\by,t)\in\mathbb{R}^{M+1}$. The base-tail cut ($C(\calT_0)=0$, i.e.\ $\cbar_{\calT_0}^\top\by \le 1$) keeps
\eqref{eq:homfull} bounded.

\subsection{Linear and quadratic numerators}
\label{sec:qpnum}
Nothing above used linearity of the numerator: the homogenisation acts on the \emph{denominator}, and the cuts \eqref{eq:homcut} are generated from a tail
ranking that never sees $f$. Only the \emph{class} of the master depends on $\boldsymbol{\Sigma}$. For $\boldsymbol{\Sigma}=0$ the quadratic term drops and \eqref{eq:homfull} is a linear program in $M{+}1$ variables. \h{For $\boldsymbol{\Sigma}\neq0$, factor $\boldsymbol{\Sigma}=\boldsymbol{L}\boldsymbol{L}^\top$, with $\boldsymbol{L}\in\mathbb{R}^{M\times r}$ and $r=\operatorname{rank}\boldsymbol{\Sigma}$. For $t>0$, introduce the epigraph variable $s$ through}
\[
  s \ge \frac{\by^\top\boldsymbol{\Sigma}\by}{2t}
  \quad\Longleftrightarrow\quad
  \|\boldsymbol{L}^\top\by\|^2 \le 2ts.
\]
\h{With $t,s\ge0$ the right-hand inequality is a rotated second-order cone, so \eqref{eq:homfull} becomes}
\begin{align}
  \max_{\by,\,t,\,s}\; \br^\top\by - s \\
  \text{s.t.} \quad \cbar_{\calT}^\top \by + C(\calT)\,t \le 1 \;\;\forall \calT,\quad \\
  \|\boldsymbol{L}^\top\by\|^2 \le 2\,t\,s, \quad
  \underline{\bx}\,t \le \by \le \overline{\bx}\,t,\quad t \ge 0,\quad s \ge 0,
  \label{eq:homqp}
\end{align}
a SOCP in $M{+}2$ variables with a single cone of dimension
$r$. Either way the master's size is independent of the scenario count $S$.

The factor $\boldsymbol{L}$ is formed once, before the first cut, and never recomputed, so it sits outside the per-cut cost entirely (Algorithm~\ref{alg:ccecp}, line~1). As $\boldsymbol{\Sigma}$ is only semidefinite we
take it from a symmetric eigendecomposition $\boldsymbol{\Sigma}=\boldsymbol{Q}\Lambda \boldsymbol{Q}^\top$, setting
$\boldsymbol{L}=\boldsymbol{Q}_+\boldsymbol{\Lambda}_+^{1/2}$ over the positive eigenvalues, rather than by Cholesky,
which assumes definiteness and can break down when $\boldsymbol{\Sigma}$ is singular; this also delivers $r$, and hence the smallest cone, directly. Both cost $O(M^3)$
once, negligible for the moderate $M$ that the $M\ll S$ regime assumes. 

\subsection{Algorithm}
The master has one constraint per tail set (combinatorially many) but only
a handful are ever active, so ECP generates them on demand;
Algorithm~\ref{alg:ccecp} states the procedure with its per-line costs. Two
points deserve emphasis. First, the ranking is performed on the full augmented
loss $\bb + \bA\bx^j$, not on the baseline $\bb$: added positions reshape the
tail, and scenarios that were unremarkable for the baseline can come to dominate
the augmented tail, so the base tail $\calT_0$ anchors only the first cut
($\bx^0=\boldsymbol0$). Second, the stopping test (line~7) is exactly the test
that this migration has settled.

\begin{algorithm}[t]
\caption{Euler Cutting-Plane (ECP)}
\label{alg:ccecp}
\begin{algorithmic}[1]
\State \textbf{Initialise:} $\bx^0 \leftarrow \boldsymbol{0}$;
  $\mathrm{cuts} \leftarrow \varnothing$;\;\; $k = (1-\alpha)S$ is the tail size;\;
  $\bl^0 \leftarrow \bb$;\;\; $\calT^0 \leftarrow$ top-$k$ indices of $\bl^0$;\;
  if $\boldsymbol{\Sigma}\neq0$, factor $\boldsymbol{\Sigma} = LL^\top$
  \hfill\textit{// top-$k$ via partial selection, $O(S)$;
  factorisation once, $O(M^3)$, outside the loop}
\For{$j = 0,1,2,\ldots$}
  \State $\cbar^j \leftarrow \frac1k\sum_{s\in\calT^j}\bA_{s\cdot}$,\;
    $C^j \leftarrow \frac1k\big(\sum_{s\in\calT^j}b_s - \sum_{s\in\calT_0}b_s\big)$
    \hfill\textit{// $O(Mk)$, closed form, from $\bb$}
  \State add cut $(\cbar^j)^\top\by + C^j t \le 1$
  \State $(\by^{j+1}, t^{j+1}) \leftarrow \arg\max\,\br^\top\by
    \;-\;\by^\top\boldsymbol{\Sigma}\by/(2t)$
    s.t.\ all cuts, $\underline{\bx}t\le\by\le\overline{\bx}t$, $t\ge0$
    \hfill\textit{// LP by warm-started dual simplex if $\boldsymbol{\Sigma}=0$; else SOCP}
  \State $\bx^{j+1} \leftarrow \by^{j+1}/t^{j+1}$;\;\;
    $\bl^{j+1} \leftarrow \bb + \bA\bx^{j+1}$;\;\;
    $\calT^{j+1} \leftarrow$ top-$k$ indices of $\bl^{j+1}$
    \hfill\textit{// $O(MS)$ matvec $+$ $O(S)$ partial selection}
  \If{$\calT^{j+1} = \calT^j$}
    \textbf{return} $\bx^{j+1}$
    \hfill\textit{// tail unchanged: allocation is tail-consistent, hence optimal}
  \EndIf
  \State carry $(\bl^{j+1},\calT^{j+1})$ forward as $(\bl^{j},\calT^{j})$
    \hfill\textit{// avoids re-ranking $\bx^{j+1}$ a second time next iteration}
\EndFor
\end{algorithmic}
\end{algorithm}

Each iteration costs $O(MS)$ for the matrix vector operation plus $O(S)$ for the top-$k$ selection --- a partial selection (\texttt{numpy's argpartition}), not a full sort, since only an unordered tail set and its mean are needed. The master re-solve (line~5) is cheap next to the ranking. For $\boldsymbol{\Sigma}=0$ it adds one row to an unchanged objective, so a dual-simplex warm start from the previous optimal basis restores optimality in a handful of pivots
rather than refactorising from scratch. For $\boldsymbol{\Sigma}\neq0$ the conic master
\eqref{eq:homqp} is re-solved by an interior-point method that cannot exploit
that basis, so the per-cut cost rises by a constant factor while remaining
independent of $S$.

\subsection{Finite termination and optimality}
\begin{assumption}[Non-degeneracy]
\label{ass:A1}
At each iterate the $k$-th and $(k{+}1)$-th largest entries of
$\bb+\bA\bx$ are distinct, so the active tail is unique. This holds with
probability one for continuous loss distributions and can be enforced by an
$\varepsilon$-perturbation otherwise.
\end{assumption}

\begin{theorem}[Finite termination at the global optimum]
\label{thm:ccecp}
Under Assumption~\ref{ass:A1}, Algorithm~\ref{alg:ccecp} terminates after
finitely many iterations --- at most the number of distinct tail sets,
$\binom{S}{k}$ --- and returns a global optimum of \eqref{eq:ratio}.
\end{theorem}

\begin{proof}
The feasible set of the full homogenised problem \eqref{eq:homfull} is the intersection of finitely many half-spaces \eqref{eq:homcut}, one per tail set. Algorithm~\ref{alg:ccecp} is Kelley's cutting-plane method
\cite{Kelley1960} on this polyhedron: at each step it solves a relaxation using a subset of the cuts, so its optimal value upper-bounds the true optimum. Let $\calT^{j}$ be the active tail at the recovered $\bx^{j+1}=\by^{j+1}/t^{j+1}$. If $\calT(\bb+\bA\bx^{j+1})=\calT^{j}$, the cut just added is the supporting hyperplane that is tight at $\bx^{j+1}$, so \eqref{eq:support} holds with equality and $\Delta\CVaR_\alpha(\bx^{j+1}) = (\cbar^{j})^\top\bx^{j+1}+C^{j} \le 1/t^{j+1}$: the point is feasible for \eqref{eq:ratio} and, maximising the objective over a relaxation, globally optimal. Otherwise the newly revealed tail $\calT(\bb+\bA\bx^{j+1})$ yields a cut violated by the current iterate (its CVaR exceeds the relaxation bound), so it differs from every previous cut and strictly shrinks the relaxation. Each non-terminating step thus consumes a distinct tail set; with $\binom{S}{k}$ tail sets of size $k$, termination is finite.
\end{proof}

Theorem~\ref{thm:ccecp} gives the ratio problem the same guarantee as the
constrained CVaR cutting plane: exact, finitely terminating, and no multiplier
iteration. The combinatorial bound is worst-case only; empirically (Section
\ref{sec:experiments}) the iterate count is single digits.

Moreover, every iterate  comes with a certificate: $\bx^{j+1}=\by^{j+1}/t^{j+1}$
satisfies the box, so its attained ratio $\phi(\bx^{j+1})$ lower-bounds
$\phi^\star$, while the master value upper-bounds it and decreases monotonically
as cuts accumulate. A certified gap is therefore available at every iterate, not
only on termination.

The proof is \emph{objective-agnostic}: it uses only that each master is a relaxation of \eqref{eq:homfull} (so its value upper-bounds the optimum), that
a tail-consistent iterate is feasible and attains that bound, and that any
other iterate reveals a violated cut and so consumes a distinct tail set.
No step refers to the objective \eqref{eq:persp} beyond its being the quantity
maximised, so the theorem covers the conic master \eqref{eq:homqp} and its
$\boldsymbol{\Sigma}=0$ linear case alike; the bound $\binom{S}{k}$ is likewise unchanged,
since it counts tail sets, not objective pieces.

\begin{remark}[Well-posedness and edge cases]
\label{rmk:edge}
Three conditions underlie the homogenisation. 
\begin{enumerate}
  \item \emph{Positive denominator:} The substitution $t = 1/\Delta\CVaR_\alpha(\bx)$ requires 
  \begin{equation}
  \Delta\CVaR_\alpha(\bx)>0,
  \end{equation} 
  which holds for any allocation that adds tail
risk. The only excluded point is the riskless direction
$\Delta\CVaR_\alpha(\bx)=0$, where the ratio is undefined. In that degenerate case the homogenised master \eqref{eq:homfull} is unbounded along $t\to\infty$,
which the solver flags directly --- it is a property of the instance, not a failure of the method, and signals that the reward can be earned at zero
marginal tail risk. 

\item \emph{Integer tail:} When $(1-\alpha)S$
is fractional, the boundary scenario carries a fractional Value-at-Risk weight; the cut then acquires the extra $z$-term of the general Rockafellar--Uryasev
representation \cite{Fabian2008}, and ECP applies unchanged with that cut.

\item \emph{A minimum $\bx$ is required when $\boldsymbol{\Sigma}\neq0$:} For a
\emph{strictly} concave numerator the box must exclude the origin, $\underline{\bx}>\boldsymbol0$; otherwise the supremum of \eqref{eq:ratio} is approached as $\bx\to\boldsymbol0$ and is not attained. The reason is the
degenerate cell of Section~\ref{sec:prelim}. Near the origin the active tail is the base tail, whose intercept vanishes, $C(\calT_0)=0$, so there $\Delta\CVaR_\alpha$ is not merely linear but positively homogeneous; along a
ray $\bx=\theta\bd$ the ratio is therefore
\begin{equation}
  \frac{f(\theta\bd)}{\Delta\CVaR_\alpha(\theta\bd)}
  \;=\; \frac{\br^\top\bd - \tfrac12\theta\,\bd^\top\boldsymbol{\Sigma}\bd}
             {\cbar_{\calT_0}^\top\bd},
  \label{eq:raydegen}
\end{equation}
which is \emph{strictly decreasing} in $\theta$ whenever $\boldsymbol{\Sigma}\bd\neq0$: the
reward and the marginal risk both scale linearly, while the quadratic penalty
is second order, so shrinking the allocation always improves the ratio. In
homogenised terms, $C(\calT_0)=0$ means the base cut $\cbar(0)^\top\by\le1$
never bounds $t$, so $t\to\infty$, the perspective term
$\by^\top\boldsymbol{\Sigma}\by/(2t)$ vanishes and the master degenerates to its $\boldsymbol{\Sigma}=0$
linear case. The implementation therefore rejects
$\underline{\bx}=\boldsymbol0$ for nonzero $\boldsymbol{\Sigma}$ rather than reporting a tolerance-dependent point for a supremum that is not attained. Any strictly positive floor $\underline{\bx}>\boldsymbol0$ bounds $\theta$ away from zero and restores attainment; the experiments of Section~\ref{sec:qpexp} use $\underline{x}=0.05$. The linear case is unaffected, since \eqref{eq:raydegen} is then constant in $\theta$.
\end{enumerate}
\end{remark}

\subsection{Relation to Dinkelbach}
For completeness, the classical alternative is Dinkelbach's algorithm
\cite{Dinkelbach1967}, shown in Algorithm~\ref{alg:dinkelbach}. It iterates
$\lambda_{j+1} = f(\bx_j) / \Delta\CVaR_\alpha(\bx_j)$ and solves
$\max_{\bx} f(\bx) - \lambda_j \Delta\CVaR_\alpha(\bx)$ each step.
Even when the inner maximisation is itself a cutting plane, the objective is
re-parameterised by an evolving $\lambda_j$, so cuts cannot simply accumulate
into one fixed program, and the method carries the convergence behaviour of the
outer fixed-point iteration.

\begin{algorithm}[t]
\caption{Dinkelbach with Rockafellar--Uryasev (baseline)}
\label{alg:dinkelbach}
\begin{algorithmic}[1]
\State \textbf{Initialise:} $\bx^0$ any strictly feasible point;
  $\lambda_0 \leftarrow f(\bx^0) / \Delta\CVaR_\alpha(\bx^0)$
\For{$j = 0,1,2,\ldots$}
  \State $\bx^{j+1} \leftarrow
    \arg\max_{\underline{\bx} \le \bx \le \overline{\bx}}\; f(\bx)
    - \lambda_j\,\Delta\CVaR_\alpha(\bx)$
    \hfill\textit{// R--U model: $M{+}S{+}1$ variables; LP if $\Sigma=0$, else QP}
  \If{$f(\bx^{j+1}) - \lambda_j\,\Delta\CVaR_\alpha(\bx^{j+1}) \le \varepsilon$}
    \State \textbf{return} $\bx^{j+1}$
    \hfill\textit{// inner value $\approx 0$: $\lambda_j$ is the optimal ratio}
  \EndIf
  \State $\lambda_{j+1} \leftarrow f(\bx^{j+1}) / \Delta\CVaR_\alpha(\bx^{j+1})$
    \hfill\textit{// update multiplier (superlinear convergence)}
\EndFor
\end{algorithmic}
\end{algorithm}

\section{Numerical Experiments}
\label{sec:experiments}

\subsection{Code and software}
All experiments were conducted on an M5 Max MacBook Pro with 64\,GB of memory. The optimisation code is written in Python~3.13 with \texttt{numpy} for the ranking and cut formation, and Gurobi~13.0.2 for every mathematical program. Both methods call the same solver through the same interface, so the comparisons that follow reflect the two formulations rather than two implementations. Each method is given the settings that suit it. For $\boldsymbol{\Sigma}=0$ the ECP master is re-solved by dual simplex with basis reuse across cut additions (\texttt{Method}~$=1$, \texttt{LPWarmStart}~$=2$); for $\boldsymbol{\Sigma}\neq0$ the conic master is left to the barrier method. The Rockafellar--Uryasev model is built once per instance and re-solved as only its objective coefficients change --- the quadratic penalty is $\lambda$-independent, so it stays a fixed part of the objective across Dinkelbach iterations --- and its quadratic variant uses primal simplex (\texttt{Method}~$=0$), which on these instances is markedly faster than the barrier method Gurobi selects automatically (Section~\ref{sec:qpexp}). Dual reductions are disabled throughout, so that an unbounded instance is reported as such rather than as infeasible-or-unbounded; this is what lets the riskless direction of Remark~\ref{rmk:edge}(i) be detected rather than inferred.

Unless stated otherwise each synthetic configuration is run over five
independently seeded instances with five timing repetitions apiece (nine for the case study of Section~\ref{sec:realdata}), and we report medians; all seeds are fixed, so each figure and table is reproducible from the accompanying scripts.

\subsection{Synthetic Problem}
\subsubsection{Setup}
\label{sec:setup}
We generate base losses $\bb$ from a factor model with $N=20$ independent exponential risk drivers, and a loss matrix $\bA$ in two regimes: \emph{correlated}, where each column shares the portfolio factor (sample correlation $\rho\approx0.7$), and \emph{mismatch}, where each column is driven by an independent factor. We use $M=8$, $\alpha=0.99$, $\bar{x}=0.3$, and rewards $\br$ drawn uniformly on $[0.04,0.12]$. Scenario counts range over $S\in\{500,\ldots,20\,000\}$ (multiples of $100$, so $k=(1-\alpha)S$ is integral). Sections~\ref{sec:results}--\ref{sec:realdata} use the linear numerator, $\boldsymbol{\Sigma}=0$; the concave quadratic case is benchmarked separately in Section~\ref{sec:qpexp}. We compare ECP against the exact optimum obtained by Dinkelbach with the full R--U inner LP. We take this as the baseline because it is the standard route for the reward-to-CVaR ratio.



\subsubsection{Results with $\boldsymbol{\Sigma} = 0$}
\label{sec:results}

\begin{figure}[t]
  \centering
  \includegraphics[width=\textwidth]{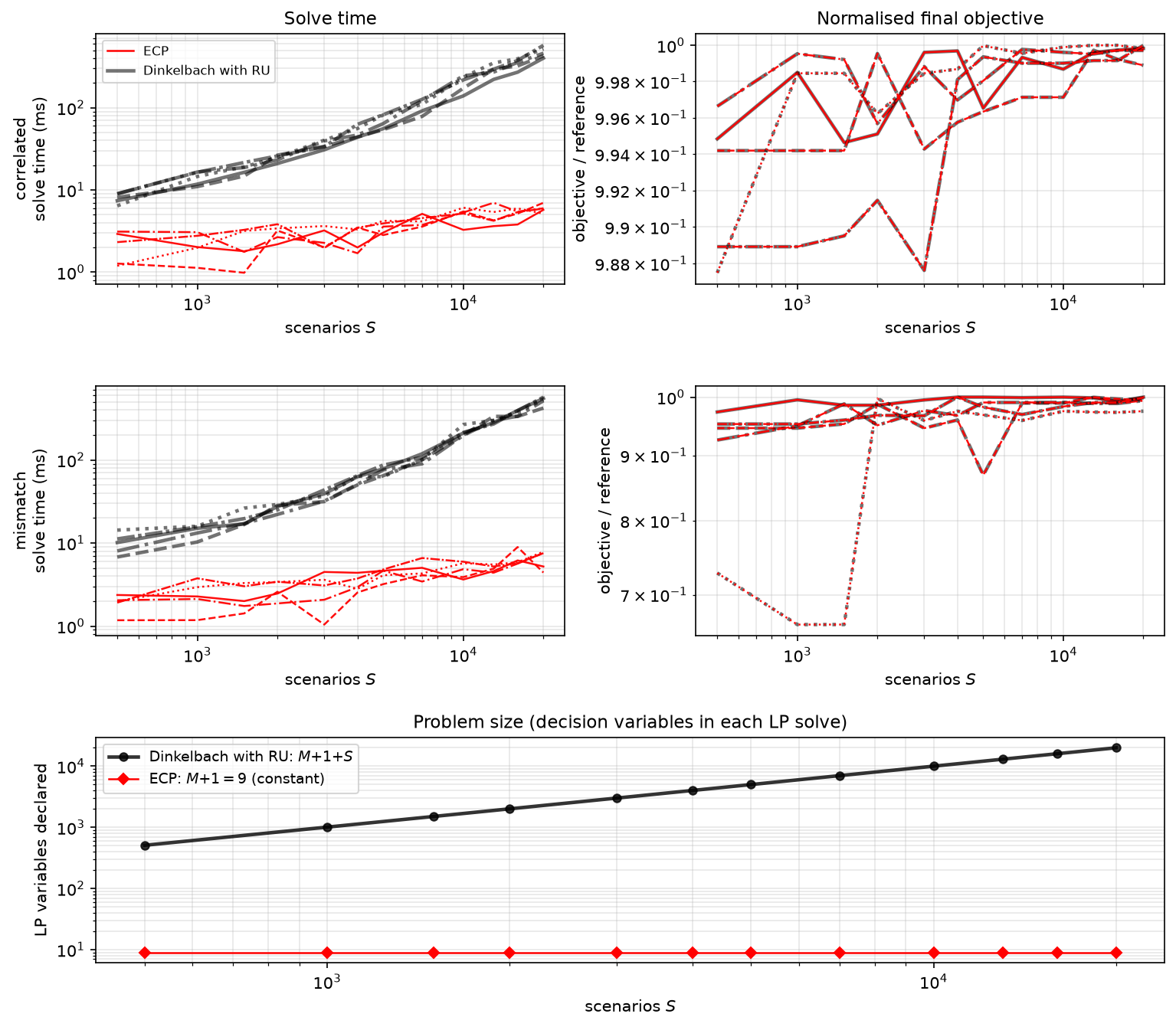}
  \caption{ECP (red) versus the exact Dinkelbach with RU  (black) algorithms on the reward-to-CVaR ratio \eqref{eq:ratio}. Five independent random instances are shown per panel (fixed seeds; each instance is one line style, matched between the red and black curves so the two methods can be compared instance by instance).
    \emph{Top/middle rows:} the correlated and mismatch regimes;
    \emph{Left column:} solve time vs.\ $S$ (log-log);
    \emph{Right column:} the final objective, normalised by each instance's $50\,000$-scenario reference optimum, and
    \emph{Bottom panel:} the number of decision variables in each LP solve.}
  \label{fig:ratio}
\end{figure}

In terms of exactness, the normalised final objective of ECP tracks that of the Dinkelbach with RU solver throughout the fresh sweep (right column of Fig.~\ref{fig:ratio}). On speed, the structural difference of the bottom panel makes ECP markedly cheaper, and the advantage widens with $S$. 

To study the effect on increasing the number of decision variables, $\boldsymbol{x}$, our next synthetic experiment fixes $S=10\,000$ and varies the number of decisions
$M$ from $2$ to $128$. The results are reported in Figure~\ref{fig:mscaling} fixes.  ECP's cut count stays in single digits ($K=5$--$8$ even at $M=128$) and its solve time grows slowly, because the homogenised master remains an $(M{+}1)$-variable LP and the closed-form Euler cut is an $O(MS)$ average; the cut count does not blow up with the decision dimension.
The Dinkelbach+R--U baseline, by contrast, pays $M{+}1{+}S$ variables at every inner solve throughout.

\begin{figure}[t]
  \centering
  \includegraphics[width=\textwidth]{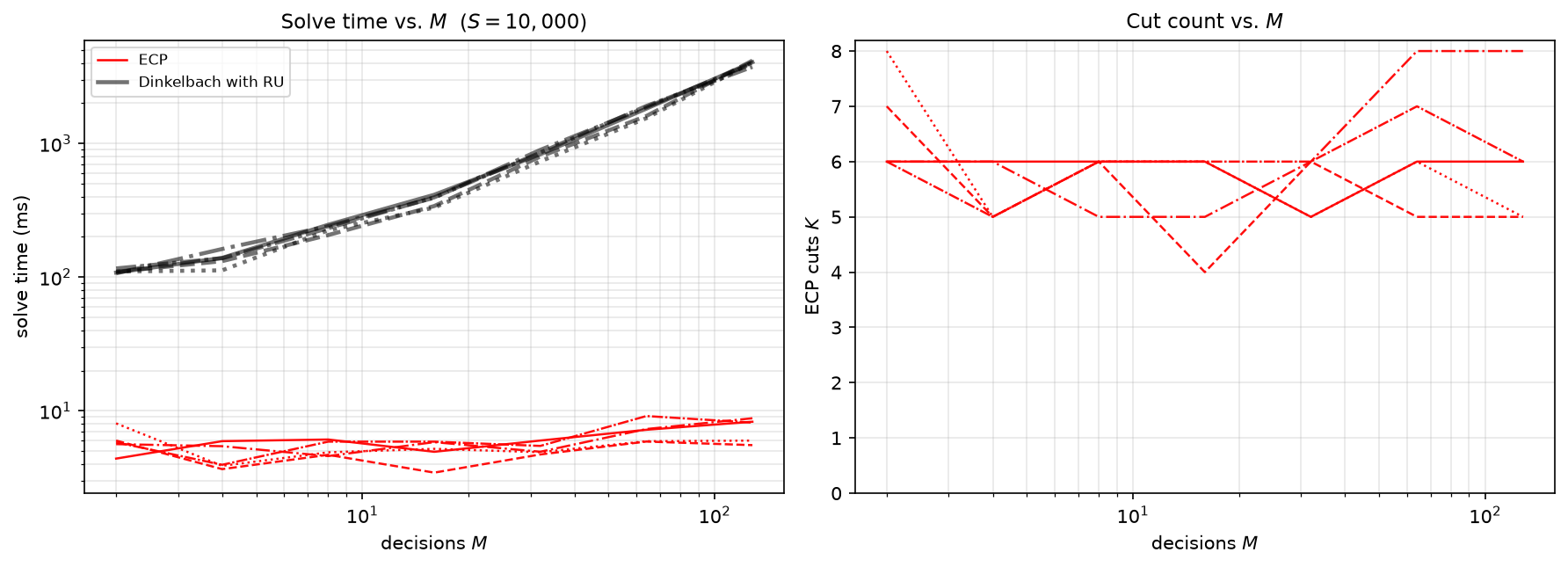}
  \caption{Scaling in the decision dimension $M$ at fixed $S=10\,000$ (mismatch regime, five fixed-seed instances per line style).
    \emph{Left:} solve time vs.\ $M$, and 
    \emph{Right:} the ECP cut count $K$ remains single-digit across the whole range.}
  \label{fig:mscaling}
\end{figure}

\subsubsection{Results with $\boldsymbol{\Sigma} \neq 0$}
\label{sec:qpexp}
We close the synthetic experiments by exercising the extension of
Section~\ref{sec:qpnum}. The setup is that of Section~\ref{sec:setup} ($M=8$, $\alpha=0.99$, mismatch regime, five fixed-seed instances) with the numerator replaced by $f(\bx)=\br^\top\bx-\tfrac12\bx^\top\boldsymbol{\Sigma}\bx$. We draw $\boldsymbol{\Sigma}=\kappa \boldsymbol{BB}^\top\!/M$ with $\boldsymbol{B}$ standard normal (full rank almost surely) and fix $\kappa$ through a dimensionless \emph{curvature} $\gamma$, defined by $\tfrac12\overline{\bx}^\top\boldsymbol{\Sigma}\,\overline{\bx} = \gamma\,\br^\top\overline{\bx}$ at the cap $\overline{\bx}=\bar x\,\boldsymbol1$. Here $\gamma$ is the fraction of the linear reward cancelled by the penalty at full allocation, and $\gamma=0$ recovers the linear problem exactly. Following Remark~\ref{rmk:edge}(iii) the box carries a minimum line size $\underline{x}=0.05$; the scenario count is capped at $S=10\,000$.

As in the linear case the Rockafellar--Uryasev model is built once and re-solved with only its linear objective coefficients
changing. Gurobi's automatic algorithm choice for a QP is a barrier method, which cannot exploit that warm basis and is a poor fit here; we therefore solve the QP baseline with primal simplex, which on these instances is uniformly faster. Figure~\ref{fig:quadratic} reports the outcome. 

\begin{figure}[t]
  \centering
  \includegraphics[width=\textwidth]{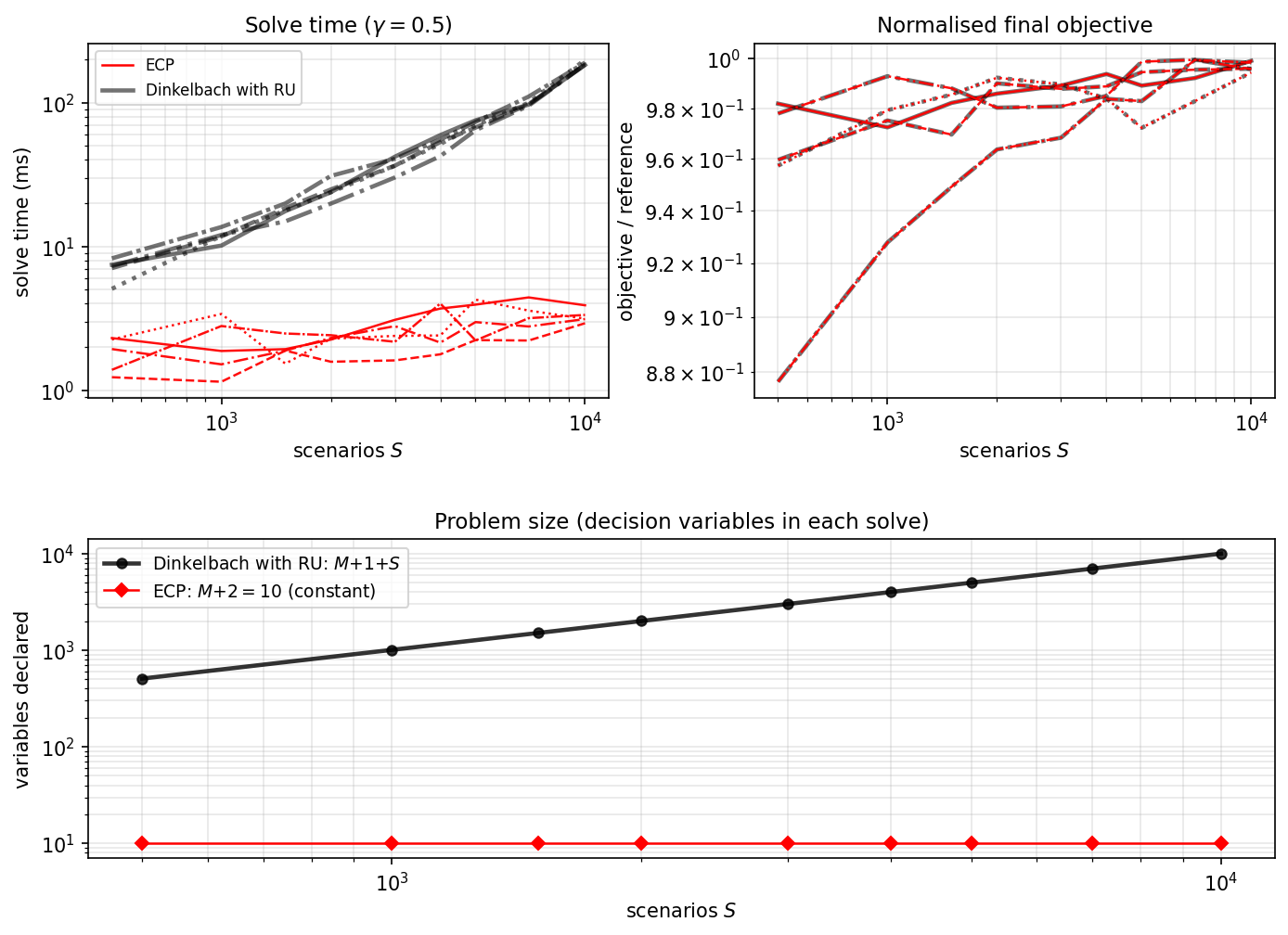}
  \caption{Concave quadratic numerator $f(\bx)=\br^\top\bx-\frac12\bx^\top\boldsymbol{\Sigma}\bx$ with $\boldsymbol{\Sigma}\succeq0$ full rank ($M=8$, $\alpha=0.99$, mismatch regime, minimum line $\underline{x}=0.05$, five fixed-seed instances per line style).
    \emph{Top left:} solve time vs.\ $S$ at curvature $\gamma=0.5$;
    \emph{Top right:} the final objective of each method, normalised by that instance's reference optimum --- the exact ratio attained on a
    $S_{\mathrm{ref}}=10\,000$ scenario set;
    \emph{Bottom:} the number of decision variables each formulation declares.}
  \label{fig:quadratic}
\end{figure}



\subsection{Real-data case study}
\label{sec:realdata}
We close with a case study on a real specialty-insurance portfolio. The base loss vector $\bb$ is a catastrophe scenario set of $S=10\,000$ simulated annual losses for wind storms for an in-force book (aggregate gross loss per year); its $1$-in-$100$ tail risk is $\CVaR_{0.99}(\bb)\approx\$196\text{M}$. Three candidate policy lines (A, B, C) are available to add; each carries its own scenario loss column, aligned to the same $10\,000$ years, and a per-unit reward $r_i = \text{premium}_i\,(1-\text{loss ratio}_i)$. The variable $\boldsymbol{x}$ represents the fraction of each policy---both its per-unit reward and its lossess---that an insurer stands to inherit. This idea of an insurer taking up a fraction of the premium and thus losses associated with a particular policy is standard within specialty insurance. 

Policy~A is by far the largest (premium $\approx\$63$M, $r_{\mathrm A}\approx \$44$M) with B and C two orders of magnitude smaller. The decision $x_i$ is the signed line size, constrained to a realistic band $x_i\in[0.01,0.10]$ (a nonzero minimum line and a $10\%$ cap), and we optimise the reward-to-marginal-CVaR ratio~\eqref{eq:ratio} at $\alpha=0.99$. As-is (stress factor $c=1$ in Table~\ref{tab:realdata}), the three lines are small relative to the book's $\$196$M tail: even at the $10\%$ cap their losses never enter the $100$ worst years, so no scenario migrates, the marginal risk is exactly linear, and the ratio is scale-invariant and maximised at a box vertex (Line~A to its cap, B and C at the floor). ECP certifies this in a \emph{single} cut. To exercise the tail-migration regime we scale the three loss columns by a stress factor $c$ (representing materially larger prospective lines, or equivalently a smaller, more concentrated book). 

For $c\ge20$ the scaled lines break into the $1$-in-$100$ tail; the denominator becomes genuinely nonlinear, the optimal Line~A allocation pulls back below its cap, and ECP accumulates several Euler cuts as the tail migrates.\footnote{Real losses carry large monetary units ($\sim\!10^8$); the homogenised master is solved on losses internally rescaled by the base risk so the LP stays well conditioned. Line sizes $\bx=\by/t$ are dimensionless and the reported ratio is evaluated in
original units, so the rescaling affects neither the solution nor the
comparison.} 

Table~\ref{tab:realdata} reports both methods across the stress range. In every case the two solvers return the \emph{same} optimum: the ratios agree to machine precision and the recovered line sizes to $|\Delta\bx^\star|<10^{-13}$ (Table~\ref{tab:realdata_x}), confirming Theorem~\ref{thm:ccecp} on real data. ECP terminates in $K=1$--$6$ cuts against $3$--$6$ Dinkelbach outer iterations, and is roughly $21$--$39\times$ faster (median wall-clock). The absolute times are only tens of milliseconds here because $M=3$ is small.

\begin{table}[t]
\centering
\caption{Real-data case study: Times are
  median wall-clock over ten repetitions.}
\label{tab:realdata}
\begin{tabular}{r r r r r r}
\toprule
 & \multicolumn{2}{c}{ECP (proposed)} & \multicolumn{2}{c}{Dinkelbach with RU} & \\
\cmidrule(lr){2-3}\cmidrule(lr){4-5}
$c$ & Objective & $t$\,(ms) & Objective & $t$\,(ms) & speedup \\
\midrule
$1$    & $43.87549$ & $1.35\pm0.11$ &  $43.87549$ & $52.58\pm1.07$ & $39.1\times$ \\
$20$   & $2.18453$ & $2.88\pm0.20$ &  $2.18453$ & $79.37\pm0.64$  & $27.5\times$ \\
$100$  & $0.43250$ & $3.42\pm0.11$ & $0.43250$ & $98.70\pm3.48$  & $29.1\times$ \\
$200$  & $0.21528$ & $4.03\pm0.35$ & $0.21528$ & $87.09\pm1.08$  & $21.6\times$ \\
\bottomrule
\end{tabular}
\end{table}

\begin{table}
\centering
\caption{Optimal line sizes $\bx^\star$ recovered by each method for the case
  study of Table~\ref{tab:realdata}.}
\label{tab:realdata_x}
\begin{tabular}{r r r r r r r}
\toprule
 & \multicolumn{3}{c}{ECP $\bx^\star$} & \multicolumn{3}{c}{Dineklbach with R--U $\bx^\star$} \\
\cmidrule(lr){2-4}\cmidrule(lr){5-7}
$c$ & A & B & C & A & B & C \\
\midrule
$1$   & $0.1000$ & $0.0100$ & $0.0100$ & $0.1000$ & $0.0100$ & $0.0100$ \\
$20$  & $0.0312$ & $0.0100$ & $0.0100$ & $0.0312$ & $0.0100$ & $0.0100$ \\
$100$ & $0.0349$ & $0.0100$ & $0.0100$ & $0.0349$ & $0.0100$ & $0.0100$ \\
$200$ & $0.0197$ & $0.0100$ & $0.0100$ & $0.0197$ & $0.0100$ & $0.0100$ \\
\bottomrule
\end{tabular}
\end{table}

\section{Conclusion}
\label{sec:conclusion}
We have shown that the reward-to-CVaR ratio problem, usually solved by a nested Dinkelbach iteration, collapses into a single cutting-plane linear program. The Charnes--Cooper homogenisation removes the reward/risk multiplier, and the convex CVaR denominator is represented by supporting hyperplanes whose coefficients are the Euler allocation vectors practitioners already compute. The resulting method has $M{+}1$ variables independent of the scenario count, terminates finitely at the global optimum, and --- with both methods solved by the same commercial solver and the baseline warm-started --- is roughly $20$--$100\times$ faster than the full Rockafellar--Uryasev route. 

Both ingredients are individually classical; their combination for a CVaR denominator appears to be new, and it gives the ratio problem the same clean guarantee long available for the constrained case. Because the construction acts entirely on the denominator, any concave numerator is handled by the same cuts and the same termination guarantee, at the cost of a second-order cone master in place of the linear one.

Open directions include a tighter iteration bound (the iterate count is single-digit and nearly dimension-free empirically, but we have no matching upper bound beyond the combinatorial one), distributionally robust variants via a Wasserstein denominator, and integer lot-size constraints through branch-and-bound on the homogenised relaxation. A further direction, worked out for the independent case in Appendix~\ref{app:multiperil}, is the aggregation of several perils---required when the losses cover not just wind (as per Section~\ref{sec:realdata}) but also earthquakes, floods, and other natural modelled hazards---whose scenario sets share no common index: the marginal-CVaR denominator is then taken over the \emph{convolution} of the per-peril loss tables, and we show the ECP cut can still be generated exactly without ever materialising the product scenario set.

\appendix
\section{Appendix}
\subsection{Exact multi-peril convolution via threshold ranking}
\label{app:multiperil}
The method extends to several perils whose scenario sets are \emph{not} aligned by a common index --- for example a wind catastrophe table and an earthquake table, each generated by its own simulation with its own year identifiers. This is standard within the disaster / catastrophe modelling community. When the perils are independent and each table is equiprobable, the aggregate (annual) loss is the \emph{convolution} of the per-peril losses, and the marginal-CVaR denominator of \eqref{eq:ratio} must be evaluated over that convolution. We show the Euler cut of Section~\ref{sec:ccecp} is still available in closed form, at essentially the single-peril cost.

We state it for two perils; the extension to a fixed number is immediate. Let peril $w$ have loss table $(\bb_w,\bA_w)\in\mathbb{R}^{S_w}\times\mathbb{R}^{S_w\times M}$ and peril $q$ have $(\bb_q,\bA_q)\in\mathbb{R}^{S_q}\times\mathbb{R}^{S_q\times M}$, so at an allocation $\bx$ the per-peril losses are $w_i(\bx)=b_{w,i}+\bA_{w,i\cdot}\bx$ and $q_j(\bx)=b_{q,j}+\bA_{q,j\cdot}\bx$. Under independence with equiprobable scenarios, every pair $(i,j)$ is an equally likely combined scenario with loss $w_i(\bx)+q_j(\bx)$, so the convolution is the Cartesian product: a scenario set of size $S=S_wS_q$ with loss vector $L_{ij}(\bx)=w_i(\bx)+q_j(\bx)$ and tail size $k=(1-\alpha)S$. Since this is again a single empirical CVaR, Algorithm~\ref{alg:ccecp} applies verbatim once its cut coefficient $\cbar$ and constant $C$ are computed over the product tail $\calT=\{(i,j):L_{ij}(\bx)\ge v_k\}$, where $v_k$ is the $k$-th largest combined loss. The only obstacle is that $S=S_wS_q$ may be enormous ($10^8$ for two tables of $10^4$ years); we now avoid it entirely with the methodology detailed below.

Fix $\bx$ and write $w_i,q_j$ for the per-peril losses. For a candidate VaR
level $v$, the number of combined scenarios exceeding it is
\begin{equation}
  N(v) \;=\; \#\{(i,j): w_i+q_j\ge v\}
        \;=\; \sum_{i=1}^{S_w}\#\{j: q_j\ge v-w_i\},
  \label{eq:countoracle}
\end{equation}
which is nonincreasing in $v$ and, having sorted $q$ once, costs
$O(S_w\log S_q)$ per evaluation. The combined VaR is the $k$-th largest pairwise
sum, recovered by bisection on the scalar $v$,
\begin{equation}
  v_k \;=\; \max\{v : N(v)\ge k\},
  \label{eq:vk}
\end{equation}
in a fixed number of steps. With $v_k$ in hand, define the per-scenario tail
multiplicities --- how many product-tail pairs each per-peril scenario belongs to,
\begin{equation}
  n_i \;=\; \#\{j: q_j\ge v_k-w_i\},
  \qquad
  m_j \;=\; \#\{i: w_i\ge v_k-q_j\}.
  \label{eq:counts}
\end{equation}
Both count the same tail set grouped differently, so
$\sum_i n_i=\sum_j m_j=|\calT|=k$. Because the combined loss is separable, the
tail sum collapses to two per-peril weighted sums,
\begin{equation}
  \sum_{(i,j)\in\calT}\!\big(w_i+q_j\big)
  \;=\;\sum_i n_i\,w_i+\sum_j m_j\,q_j,
  \label{eq:collapse}
\end{equation}
and the identical grouping applied to the loss-table rows gives the Euler cut
coefficient and constant:
\begin{align}
  \cbar \;&=\; \frac1k\Big(\textstyle\sum_i n_i\,\bA_{w,i\cdot}
                          +\sum_j m_j\,\bA_{q,j\cdot}\Big),
  \label{eq:convcbar}\\
  C \;&=\; \frac1k\Big(\textstyle\sum_i n_i\,b_{w,i}+\sum_j m_j\,b_{q,j}\Big)
          \;-\;\CVaR_\alpha^{\mathrm{conv}}(\bb_w,\bb_q),
  \label{eq:convC}
\end{align}
where the subtracted term is the base-tail constant, i.e.\ the convolution CVaR
of the two baseline tables (evaluated by the same routine at $\bx=\boldsymbol0$),
so that $C(\calT_0)=0$ as in Section~\ref{sec:prelim}.

\begin{algorithm}[t]
\caption{Exact convolution Euler cut for two independent perils}
\label{alg:convcut}
\begin{algorithmic}[1]
\Require per-peril losses $w_i=\bb_{w,i}+\bA_{w,i\cdot}\bx$,
  $q_j=\bb_{q,j}+\bA_{q,j\cdot}\bx$;\; $k=(1-\alpha)S_wS_q$
\State sort $q$ ascending \hfill\textit{// $O(S_q\log S_q)$; reused each bisection step}
\State bisect $v$ on $[\,w_{\min}+q_{\min},\,w_{\max}+q_{\max}\,]$ using
  $N(v)$ of \eqref{eq:countoracle} to solve \eqref{eq:vk} for $v_k$
  \hfill\textit{// $O(S_w\log S_q)$ per step}
\State $n_i\leftarrow\#\{j:q_j\ge v_k-w_i\}$,\;
  $m_j\leftarrow\#\{i:w_i\ge v_k-q_j\}$ \hfill\textit{// eq.~\eqref{eq:counts}}
\State \textbf{return} $\cbar,C$ from \eqref{eq:convcbar}--\eqref{eq:convC}
  \hfill\textit{// feed to Algorithm~\ref{alg:ccecp} unchanged}
\end{algorithmic}
\end{algorithm}

Each cut therefore costs $O\big((S_w+S_q)\log(S_w S_q)\big)$ --- the same $O(S\log S)$ ranking as the single-peril case --- while the LP stays at $M{+}1$
variables. Finite termination is inherited unchanged: the convolution feasible set is again an intersection of finitely many half-spaces, one per distinct product-tail set (at most $\binom{S_wS_q}{k}$), so the proof of
Theorem~\ref{thm:ccecp} applies with $\binom{S}{k}$ replaced by $\binom{S_wS_q}{k}$.

Two caveats delimit the scope. First, \eqref{eq:countoracle} exploits independence and equal scenario weights; a dependence structure between perils, or non-uniform scenario probabilities, breaks the product form. Second, the $P$-peril generalisation turns the scalar bisection \eqref{eq:vk} into a $P$-dimensional threshold search, which loses its edge as $P$ grows. Outside these conditions the general fallback is to fix a Monte-Carlo sample of the convolution once, before optimisation, and run Algorithm~\ref{alg:ccecp} on that frozen synthetic table --- correct for arbitrary dependence (sampled from a copula) and any number of perils, at the cost of $O(1/\sqrt{k})$ tail sampling
error.

\subsection{Adding a CVaR budget constraint}
\label{app:budget}

Practitioners rarely maximise a reward-to-risk ratio in isolation: the allocation must also respect an absolute tail-risk budget on the resulting book. We show that a constraint of the form
\begin{equation}
  \CVaR_\alpha(\bb + \bA\bx) \;\le\; B
  \label{eq:budget}
\end{equation}
is absorbed into the Euler cutting-plane method at the cost of \emph{one extra LP row per iteration and no extra variables}, because the constraint is generated
from the same tail ranking that already produces the objective cut. We then contrast this with the Rockafellar--Uryasev route, and make precise when that
route does, and does not, pay for additional scenario variables.

The constraint functional in \eqref{eq:budget} is the \emph{same} empirical CVaR of the \emph{same} augmented loss $\bb+\bA\bx$ that appears in the ratio denominator. Using the max-of-affines representation of Section~\ref{sec:prelim},
$\CVaR_\alpha(\bb+\bA\bx)=\max_\calT\{\bar b_\calT + \cbar_\calT^\top\bx\}$ with
$\bar b_\calT := \tfrac1k\sum_{s\in\calT}b_s$, the budget \eqref{eq:budget} is
equivalent to the (combinatorial) family of linear inequalities
\begin{equation}
  \bar b_\calT + \cbar_\calT^\top\bx \;\le\; B \qquad \text{for all tail sets }\calT.
  \label{eq:budgetfamily}
\end{equation}
These are handled by exactly the same on-demand generation as the denominator cuts. Homogenising through $t=1/\Delta\CVaR_\alpha(\bx)>0$, $\by=t\bx$ and
multiplying \eqref{eq:budgetfamily} by $t$ gives, for every tail set $\calT$, the single homogeneous row
\begin{equation}
  \cbar_\calT^\top\by + \big(\bar b_\calT - B\big)\,t \;\le\; 0.
  \label{eq:budgethom}
\end{equation}
Crucially, $\bar b_\calT$ is not a new quantity: from the intercept definition \eqref{eq:support}, $\bar b_\calT = C(\calT) + \CVaR_\alpha(\bb)$, so \eqref{eq:budgethom} reuses the coefficient $\cbar_\calT$ and constant $C(\calT)$ that Algorithm~\ref{alg:ccecp} already forms on line~3, together with the fixed base risk $\CVaR_\alpha(\bb)$ and the budget $B$. Concretely, the only change to
Algorithm~\ref{alg:ccecp} is to append, alongside the objective cut on line~4, the budget row
\begin{equation}
  (\cbar^j)^\top\by + \big(C^j + \CVaR_\alpha(\bb) - B\big)\,t \;\le\; 0.
  \label{eq:budgetline}
\end{equation}
Both cuts key off the identical active tail $\calT^j$ of $\bb+\bA\bx^j$: the $O(S\log S)$ ranking and the $O(MS)$ coefficient formation are performed \emph{once} and shared. The LP stays at $M{+}1$ variables and $S$ never enters its dimension.

The augmented feasible set is still an intersection of finitely many half-spaces, so Algorithm~\ref{alg:ccecp}
remains Kelley's method on a polyhedron and Theorem~\ref{thm:ccecp} applies with $\binom{S}{k}$ replaced by (at most) $2\binom{S}{k}$. The stopping test is unchanged: because objective and budget cuts share the tail of $\bb+\bA\bx$, when the tail stabilises ($\calT(\bb+\bA\bx^{j+1})=\calT^j$) both \eqref{eq:support} and \eqref{eq:budgetfamily} hold with equality at $\bx^{j+1}$, so the denominator linearisation \emph{and} the budget are exact simultaneously. One feasibility condition delimits scope: with a nonnegative box $\bx\ge\boldsymbol0$ and nonnegative marginal losses, the tail can only grow, so
\eqref{eq:budget} is attainable only for $B\ge\CVaR_\alpha(\bb)$; then the base point $\bx=\boldsymbol0$ ($\by=\boldsymbol0$) satisfies \eqref{eq:budgethom} for
every $\calT$ and the homogenised feasible set is nonempty.

\section*{Declarations}

\subsection*{Ethics approval and consent to participate}
\h{Not applicable. This study involves no human participants or animals.}

\subsection*{Consent for publication}
\h{Not applicable. The manuscript contains no individual person's data in any form.}

\subsection*{Funding}
\h{The authors received no specific funding for this work.}

\subsection*{Availability of data and materials}
\h{The synthetic experiments are generated by the code in the project repository, \url{https://github.com/Ki-Insurance/marginal-tvar-approx}, which also contains the scripts that reproduce the figures and tables. The real-data case study uses a proprietary catastrophe loss set and policy-line data belonging to Ki Insurance. These are commercially confidential and are not publicly available.}

\subsection*{Acknowledgements}
\h{Not applicable.}


\end{document}